\documentclass[amstex,12pt,reqno]{amsart}
\usepackage{amsmath,amsfonts,amssymb,amsthm,enumerate,hyperref,multicol}

\usepackage{graphicx}

\newtheorem{lemma}{Lemma}
\newtheorem{thm}{Theorem}

\newtheorem{remark}{Remark}
\newtheorem{tb}{Table}

\numberwithin{equation}{section}
\begin{document}

\leftline{ \scriptsize \it  }
\title[]
{Approximation of Discontinuous Functions by Kantorovich Exponential Sampling Series}
\maketitle

\begin{center}
{\bf A. Sathish Kumar$^1,$ Prashant Kumar$^2$ and P. Devaraj$^3$}\\

\vskip0.2in
$^{1,2}$Department of Mathematics, Visvesvaraya National Institute of Technology Nagpur, Nagpur, Maharashtra-440010, India\\
\verb"mathsatish9@gmail.com" and \verb''pranwd92@gmail.com

\vskip0.2in
$^3$School of Mathematics, Indian Institute of Science Education and Research, Thiruvananthapuram, India.\\
\verb"devarajp@iisertvm.ac.in"
\end{center}

\begin{abstract}
The Kantorovich exponential sampling series $I_{w}^{\chi}f$ at jump discontinuities of the bounded measurable signal $f:\mathbb{R}^{+} \rightarrow \mathbb{R}$ has been analysed. A representation lemma for the series $I_{w}^{\chi}f$ is established and using this lemma certain approximation theorems for discontinuous signals are proved.
The degree of approximation in terms of logarithmic modulus of smoothness for the series $I_{w}^{\chi}f$ is studied. Further a linear prediction of signals based on past sample values has been obtained. Some numerical simulations are performed to validate the approximation of discontinuous signals $f$ by $I_{w}^{\chi}f.$
\\
\vskip0.001in

\end{abstract}

\section{Introduction}

\subsection{Preliminaries and Basic Assumptions}
First we give some basic definitions and notations. Let $L^{\infty}(\mathbb{R}^+)$ denote the set of all Lebesgue measurable essentially bounded functions on the set of all positive real numbers $\mathbb{R}^+$. For $c \in \mathbb{R}$, we define
$$X_c=\{f :\mathbb{R}^+\rightarrow\mathbb{C}:f(\cdot)(\cdot)^{c-1}\in L^1(\mathbb{R}^+)\}$$ and the norm on $X_c$ is given by
$$\Vert f \Vert_{X_c}=\Vert f(\cdot)(\cdot)^{c-1} \Vert_1=\int_0^{+\infty} |f(v)|v^{c-1}dv.$$
The Mellin transform of a function  $f \in X_c$ is defined as
$$\widehat{[f]_{M}}(s):= \int_0^{+\infty} v^{s-1}f(v)\ dv \ , \,  \ (s = c + it, t \in \mathbb{R})$$
and that it is said to be Mellin band-limited on $[-\kappa, \kappa],$ if $\widehat{[f]_{M}}(c+it)=0$ for all $|t|>\kappa,$ where $\kappa$ is some positive real number.
Butzer and Jansche in \cite{butzer5} have initiated the study of exponential sampling. For a suitable kernel $ \chi :\mathbb{R}^{+} \rightarrow \mathbb{R}$ satisfying certain conditions, Bardaro et.al. \cite{bardaro7} have analysed the uniform and point-wise convergence of the following generalized sampling series for bounded continuous functions:
\begin{equation} \label{classical}
(S_{w}^{\chi}f)(t)= \sum_{k=- \infty}^{+\infty} \chi(e^{-k} t^{w}) f( e^{\frac{k}{w}})
\end{equation}
The above sampling series $S_{w}^{\chi}f$ is well defined for $f\in L^{\infty}(\mathbb{R}^{+}).$ In \cite{bardaro11}, the exponential sampling series (\ref{classical}) was studied for functions in Mellin-Lebesgue spaces. The Mellin theory received a lot of attentions due to analysis by Mamedov \cite{mamedeo}. Further it was developed by Butzer et.al. in various research articles (see \cite{butzer3} -\cite{ butzer7}). Some of the work which are related to Mellin transform and Mellin approximation can be seen in \cite{bardaro1,bardaro9,bardaro2,bardaro3}. Bardaro et.al. \cite{bardaro3} analysed the analogue of Paley-Wiener theorem for Mellin transform and pointed out that it is different from that of Fourier transform. The exponential sampling series is used in certain areas like light scattering, Fraunhofer diffraction and radio astronomy, etc (see \cite{bert, cas, gori, ost}).

Let $ \chi :\mathbb{R}^{+} \rightarrow \mathbb{R}.$ For $\nu \in \mathbb{N}_{0}:=\mathbb{N}\cup \{0\},$ the algebraic and absolute moment of order $\nu$ are defined by
$$ m_{\nu}(\chi,u):= \sum_{k= - \infty}^{+\infty}  \chi(e^{-k} u) (k- \log u)^{\nu}, \hspace{0.5cm} \forall \ u \in \mathbb{R}^{+}$$
and
$$ M_{\nu}(\chi,u):= \sum_{k= - \infty}^{+\infty}  |\chi(e^{-k} u)| |k- \log u|^{\nu},  \hspace{0.5cm} \forall \ u \in \mathbb{R}^{+}.$$
We define $ \displaystyle M_{\nu}(\chi):= \sup_{u \in \mathbb{R}^{+}} M_{\nu}(\chi,u). $ We say that $ \chi$ is kernel if it satisfies the following conditions:
\begin{itemize}
\item[(i)] $\chi\in L^{1}(\mathbb{R}^+)$ and $\chi$ is bounded on $\left[\dfrac{1}{e},e\right],$

\item[(ii)] for every $ u \in \mathbb{R}{^{+}},$ $\displaystyle  \sum_{k=- \infty}^{+\infty} \chi(e^{-k}u) =1,$

\item[(iii)]for some $\nu>0,$  $\displaystyle M_{\nu}(\chi,u)=\sup_{u \in \mathbb{R}^{+}}  \sum_{k= - \infty}^{+\infty}  |\chi(e^{-k} u)| |k- \log u|^{\nu}<+\infty.$

\end{itemize}

\subsection{Kantorovich Exponential Sampling Series }

The exponential sampling series $S_{w}^{\chi}f$ is used to approximate a non-Mellin band limited function $f$ using its values at the nodes $ (e^{\frac{k}{w}}).$ Measuring the exact value of the sample is difficult and measurement process depends on the aperture device used for sampling. A suitable mathematical model for such a measurement process is replacing $ f( e^{\frac{k}{w}})$ by the local average of signal $f$ on the interval $ \left[e^{\frac{k}{w}}, e^{\frac{k+1}{w}} \right].$ This motivates the consideration of Kantorovich version of the exponential sampling series. In 2020, Kantorovich exponential sampling series $S_{w}^{\chi}f$ was introduced and their inverse and direct approximation results have been studied in \cite{SKB}. For $ f: \mathbb{R}^{+} \rightarrow \mathbb{R},$ the Kantorovich exponential sampling series of $f$ is defined by (see \cite{SKB})
\begin{equation} \label{main}
(I_{w}^{\chi}f)(x)= \sum_{k= - \infty}^{+\infty} \chi(e^{-k} x^{w})\  w \int_{\frac{k}{w}}^{\frac{k+1}{w}} f(e^{u})\  du.
\end{equation}
Under the conditions (i) (ii) and (iii) on the kernel $\chi,$ it is clear that $I_{w}^{\chi}f$  is well defined for $f \in L^{\infty}(\mathbb{R}^{+})$ at every $x \in \mathbb{R}^{+}.$
Direct convergence results and the improved order of approximation for Kantorovich exponential sampling series has been analysed in \cite{SKB1}. Further, the Durrmeyer type modification of exponential sampling series (\ref{classical}) is considered in \cite{Durr} and \cite{SKB2}. So far the approximation of continuous functions by Kantorovich exponential sampling series $I_{w}^{\chi}f$  has been studied. The Kantorovich exponential sampling series $I_{w}^{\chi}f$ for discontinuous signals has not been analysed. Motivated by Butzer et.al.\cite{discont} and \cite{Maria}, we study the behaviour of Kantorovich exponential sampling series $I_{w}^{\chi}f$ at jump discontinuity of bounded measurable signal $f.$

\section{Approximation of Discontinuous Functions by $I_{w}^{\chi}f$}

We analyse the behaviour of the Kantorovich exponential sampling series $I_{w}^{\chi}f$ as  $w \rightarrow \infty$ at a jump discontinuity of $f$, i.e., at a point where the limits
$$ f(t+0):=\lim_{y \rightarrow 0^{+}}f(t+y),$$ and $$ f(t-0):=\lim_{y \rightarrow 0^{+}}f(t-y)$$
exist and are different. We denote $\psi_{\chi}^{+}$ and $\psi_{\chi}^{-}$ two functions from $\mathbb{R}^{+}$ to $\mathbb{R}$ which are defined by
$$\psi_{\chi}^{+}(u):=\sum _{k<\log u}\chi(ue^{-k}),$$
and
$$ \psi_{\chi}^{-}(u):=\sum _{k>\log u}\chi(ue^{-k}),$$
where $\chi$ is a kernel function and one can easily verify that $\psi_{\chi}^{+}(u)$ and $\psi_{\chi}^{-}(u)$ are recurrent functions with fundamental interval $[1, e]$. We first prove the following representation lemma for the series $I_{w}^{\chi}f.$ Throughout this section we assume that $f(t+0)$ and $f(t-0)$ exist and are finite.

\begin{lemma}\label{lemma1}
Let $f:\mathbb{R}^{+}\rightarrow\mathbb{R}$ be a bounded measurable signal and $t\in\mathbb{R}^{+}$ be fixed. Let $h_{t}:\mathbb{R}^{+}\rightarrow\mathbb{R}$ be defined by
\begin{eqnarray*} \label{M3}
h_{t} (x) =
\left\{
\begin{array}{ll}
{f(x)-f(t-0),} & \mbox{if} \ x <t \\

{f(x)-f(t+0),} &\mbox{if} \  x >t\\

{0,}  &\mbox{if,} \ x=t.
\end{array}
\right.
\end{eqnarray*}
Then, there holds:
\begin{eqnarray*}
(I_{w}^{\chi}f)(t)&=&(I_{w}^{\chi}h_{t})(t)+f(t-0)+\left[\chi(e^{w\log t-\lfloor w\log t \rfloor} )+\psi_{\chi}^{-}(t^{w})\right](f(t+0)-f(t-0)) \\&&-\chi\left(e^{w\log t-\lfloor w\log t \rfloor} \right) \left(w\log t- \lfloor w\log t \rfloor\right)(f(t+0)-f(t-0)),
\end{eqnarray*}
if $ w\log(t)\notin\mathbb{Z},$ where $\lfloor. \rfloor$ denotes the integer part of a given number and
\begin{eqnarray*}
(I_{w}^{\chi}f)(t)=(I_{w}^{\chi}h_{t})(t)+f(t-0)+\left(\chi(1)+\psi_{\chi}^{-}(t^{w})\right)[f(t+0)-f(t-0)],
\end{eqnarray*}
if $w\log(t)\in\mathbb{Z},$ $w>0.$
\end{lemma}

\begin{proof}
Let $w\log(t)\notin \mathbb{Z},$ and $w>0.$ Then, we have
\begin{eqnarray*}
  (I_{w}^{\chi}h_{t})(t)&=&\sum _{k<\lfloor w\log t \rfloor}\chi(e^{-k} t^{w})w\int_{\frac{k}{w}}^{\frac{k+1}{w}}(f(e^{u})-f(t-0))du\\
  \\&&+\sum _{k>\lfloor w\log t \rfloor}\chi(e^{-k} t^{w})w\int_{\frac{k}{w}}^{\frac{k+1}{w}}(f(e^{u})-f(t+0)) du\\
  \\&&+\chi\left(e^{w\log t-\lfloor w\log t \rfloor} \right) w\left( \int_{\frac{\lfloor w\log t \rfloor}{w}}^{\log t}(f(e^{u})-f(t-0))du+\int_{\log t}^{\frac{\lfloor w\log t \rfloor+1}{w}}(f(e^{u})-f(t+0))du\right) \\
  &=&\sum _{k<\lfloor w\log t \rfloor}\chi(e^{-k} t^{w})w\int_{\frac{k}{w}}^{\frac{k+1}{w}}f(e^{u})du-f(t-0)\sum _{k<\lfloor w\log t \rfloor}\chi(e^{-k} t^{w})\\
  \\&&+\sum _{k>\lfloor w\log t \rfloor}\chi(e^{-k} t^{w})w\int_{\frac{k}{w}}^{\frac{k+1}{w}}f(e^{u})du-f(t+0)\sum _{k>\lfloor w\log t \rfloor}\chi(e^{-k} t^{w})\\
  \\&&+\chi\left(e^{w\log t-\lfloor w\log t \rfloor} \right) w\int_{\frac{\lfloor w\log t \rfloor}{w}}^{\frac{\lfloor w\log t \rfloor +1}{w}}f(e^{u})du
-\chi\left(e^{w\log t-\lfloor w\log t \rfloor} \right)   \\&&\times w\left(f(t-0)\left(\log t-{\frac{\lfloor w\log t \rfloor}{w}}\right)+f(t+0)\left({\frac{\lfloor w\log t \rfloor +1}{w}}-\log t\right)\right)\\
  &=& (I_{w}^{\chi}f)(t)-f(t-0)\sum _{k<\lfloor w\log t \rfloor}\chi(e^{-k} t^{w})-f(t+0)\sum _{k>\lfloor w\log t \rfloor}\chi(e^{-k} t^{w})\\\
  \\&&-\chi\left(e^{w\log t-\lfloor w\log t \rfloor} \right) \left[f(t-0)(w\log t-{{\lfloor w\log t \rfloor}})+f(t+0)({\lfloor w\log t \rfloor +1}-w\log t)\right].
\end{eqnarray*}
  Adding and subtracting $\displaystyle f(t-0)\sum _{k\geq \lfloor w\log(t) \rfloor}\chi(e^{-k} t^{w})$ in the above equation, we obtain
 \begin{eqnarray*}
  (I_{w}^{\chi}f)(t)&=&(I_{w}^{\chi}h_{t})(t)+f(t-0)\sum _{k\in \mathbb{Z}}\chi(e^{-k} t^{w})-f(t-0)\sum _{k>\lfloor w\log t \rfloor}\chi(e^{-k} t^{w})\\
  \\&&-f(t-0)\chi\left(e^{w\log t-\lfloor w\log t \rfloor} \right) +f(t+0)\chi\left(e^{w\log t-\lfloor w\log t \rfloor} \right) +f(t+0)\sum _{k>\lfloor w\log t \rfloor}\chi(e^{-k} t^{w})\\
  \\&&+\chi\left(e^{w\log t-\lfloor w\log t \rfloor} \right) \Big[ f(t-0)(w\log t-{{\lfloor w\log t \rfloor}})+f(t+0)({\lfloor w\log t \rfloor }-w\log t)\Big].
   \end{eqnarray*}
As $w\log t$ is not an integer, one can easily verify that
   $$ \psi_{\chi}^{-}(t^{w})=\sum _{k>\lfloor w\log t \rfloor}\chi(e^{-k} t^{w})$$ and in view of condition (ii) we obtain
\begin{eqnarray*}
(I_{w}^{\chi}f)(t)&=&(I_{w}^{\chi}h_{t})(t)+f(t-0)+\left[\chi\left(e^{w\log t-\lfloor w\log t \rfloor} \right)+\psi_{\chi}^{-}(t^{w})\right](f(t+0)-f(t-0)) \\&&-\chi\left(e^{w\log t-\lfloor w\log t \rfloor} \right) \left(w\log t- \lfloor w\log t \rfloor\right)(f(t+0)-f(t-0)).
\end{eqnarray*}
Now we analyse the case that $w\log(t)\in \mathbb{Z},$ and $w>0.$ We have
\begin{eqnarray*}
  (I_{w}^{\chi}h_{t})(t)&=&\sum _{k<w\log t }\chi(e^{-k} t^{w})w\int_{\frac{k}{w}}^{\frac{k+1}{w}}(f(e^{u})-f(t-0))du\\
  \\&&+\sum _{k\geq w\log t }\chi(e^{-k} t^{w})w\int_{\frac{k}{w}}^{\frac{k+1}{w}}(f(e^{u})-f(t+0)) du.
\end{eqnarray*}
Adding and subtracting $\displaystyle f(t-0)\sum _{k\geq  w\log(t)}\chi(e^{-k} t^{w})$ and using the second condition on the kernel, we get
 \begin{eqnarray*}
  (I_{w}^{\chi}f)(t)&=&(I_{w}^{\chi}h_{t})(t)+f(t-0)+(f(t+0)-f(t-0))\sum _{k\geq  w\log(t)}\chi(e^{-k} t^{w})\\
  &=&(I_{w}^{\chi}h_{t})(t)+f(t-0)+\left(\chi(1)+\psi_{\chi}^{-}(t^{w})\right)[f(t+0)-f(t-0)].
    \end{eqnarray*}
\end{proof}

We recall the following basic convergence result for continuous functions proved in \cite{SKB}.
\begin{thm}\label{t1}
Let $ f:\mathbb{R}^{+} \rightarrow \mathbb{R} $ be a bounded function . Then, $(I_w^{\chi}{f})(t)$ converges to  $f(t)$ at any point $t$ of continuity. Further,
 if $f\in C(\mathbb{R}^{+}),$ then
$$ \lim_{ w \rightarrow \infty} \| f- I_w^{\chi}{f}\|_{\infty} = 0.$$
\end{thm}

Now we show that the non-removable jump discontinuity function can be approximated by Kantorovich exponential sampling series $I_w^{\chi}f$ for $w\log(t)\in \mathbb{Z}.$

\begin{thm}\label{t2}
Suppose that $f$ has a non-removable jump discontinuity at $t\in\mathbb{R}^{+}$ and that $\alpha \in \mathbb{R}.$ Then, the following are equivalent:
\begin{enumerate}
\item[(i)] $\displaystyle \lim_{\stackrel{w \rightarrow \infty}{w\log(t)\in \mathbb{Z}}}
(I_w^{\chi}f)(t)=[\chi(1)+\alpha] f(t+0)+[1-\alpha-\chi(1)]f(t-0), \hspace{0.25cm}$\\

\item[(ii)] $\psi_{\chi}^{-}(1):=\alpha,$  \\

\item[(iii)] $\psi_{\chi}^{+}(1):=1-\alpha-\chi(1). $
\end{enumerate}
\end{thm}

\begin{proof}
Let $w\log(t)\in \mathbb{Z}.$ In view of Lemma \ref{lemma1}, we have
\begin{eqnarray*}
  (I_{w}^{\chi}f)(t)=(I_{w}^{\chi}h_{t})(t)+f(t-0)+(\chi(1)+\psi_{\chi}^{-}(t^{w}))[f(t+0)-f(t-0)],
   \end{eqnarray*}
for any $w>0.$ Since $h_{t}$ is bounded and continuous at zero and applying Theorem \ref{t1}, we get
$$\displaystyle\lim_{w \rightarrow \infty }(I_{w}^{\chi}h_{t})(t)=0.$$ Hence, we have
\begin{eqnarray*}
\displaystyle \lim_{\stackrel{w \rightarrow \infty}{w\log(t)\in \mathbb{Z}}}
(I_w^{\chi}f)(t)=f(t-0)+\left(\chi(1)+\displaystyle \lim_{\stackrel{w \rightarrow \infty}{w\log(t)\in \mathbb{Z}}} \psi_{\chi}^{-}(t^{w})\right) [f(t+0)-f(t-0)].
\end{eqnarray*}
Since $\psi_{\chi}^{-}$ is recurrent function with fundamental domain $[1, e],$ we get
$$\psi_{\chi}^{-}(t^{w})=\psi_{\chi}^{-}(1), \,\ \forall w, t \,\ \mbox{such that}  \,\ w\log(t)\in \mathbb{Z}.$$
Thus, we obtain
\begin{eqnarray*}
\displaystyle \lim_{\stackrel{w \rightarrow \infty}{w\log(t)\in \mathbb{Z}}}
(I_w^{\chi}f)(t)=[\chi(1)+\psi_{\chi}^{-}(1)] f(t+0)+[1-\psi_{\chi}^{-}(1)-\chi(1)]f(t-0).
\end{eqnarray*}

\indent $\mbox{Now}\,\  (i) \Longleftrightarrow
[\chi(1)+\alpha] f(t+0)+[1-\alpha-\chi(1)]f(t-0)$
\begin{eqnarray*}
&=&[\chi(1)+\psi_{\chi}^{-}(1)] f(t+0)+[1-\psi_{\chi}^{-}(1)-\chi(1)]f(t-0)\\
&\Longleftrightarrow& \psi_{\chi}^{-}(1)(f(t+0)-f(t-0))=\alpha(f(t+0)-f(t-0))\\
&\Longleftrightarrow& \psi_{\chi}^{-}(1)=\alpha\\
&\Longleftrightarrow& (ii)\,\ \mbox{holds}.
 \end{eqnarray*}
   Since $\displaystyle\sum_{k=- \infty}^{+\infty} \chi(e^{-k} t^{w})=1,$ we have
   $$\psi_{\chi}^{+}(1)=1-{\chi}(1)-\psi_{\chi}^{-}(1).$$
   This implies that $(ii)\Longleftrightarrow(iii).$ Hence, the proof is completed.
\end{proof}

Next in order to obtain the convergence of Kantorovich exponential sampling series $I_w^{\chi}f$ at a non-removable jump discontinuity at $t\in\mathbb{R}^{+}$ when $w\log(t)\notin \mathbb{Z}$ an extra condition on the kernel is required. We prove the following theorem.

\begin{thm}\label{t3}
Let $f:\mathbb{R}^{+}\rightarrow\mathbb{R}$ be a bounded measurable signal with a non-removable jump discontinuity at $t\in\mathbb{R}^{+}$ and let $\alpha \in \mathbb{R}.$ Suppose that $\chi(u)=0,$ for every $u\in(1,e).$ Then, the following statements are equivalent:
\begin{enumerate}
\item[(i)] $\displaystyle \lim_{\stackrel{w \rightarrow \infty}{w\log(t)\notin \mathbb{Z}}} (I_w^{\chi}f)(t)=\alpha f(t+0)+(1-\alpha)f(t-0),$\\

\item[(ii)] $\psi_{\chi}^{-}(u):=\alpha, \hspace{0.5cm} u \in(1, e) $ \\

\item[(iii)]  $\psi_{\chi}^{+}(u):=1-\alpha, \hspace{0.5cm} u\in(1, e).$

\end{enumerate}
\end{thm}

\begin{proof}
Using Lemma \ref{lemma1} and the condition that $\chi(u)=0,$ for every $u\in(1,e),$ we obtain
\begin{eqnarray*}
 \displaystyle \lim_{\stackrel{w \rightarrow \infty}{w\log(t)\notin \mathbb{Z}}} (I_{w}^{\chi}f)(t)
 =f(t-0)+\left(\displaystyle \lim_{\stackrel{w \rightarrow \infty}{w\log(t)\notin \mathbb{Z}}}\psi_{\chi}^{-}(t^{w})\right)[f(t+0)-f(t-0)].
   \end{eqnarray*}
   \begin{eqnarray*}
   (i)&\Longleftrightarrow &
\alpha f(t+0)+(1-\alpha)f(t-0)=f(t-0)+\left(\displaystyle \lim_{\stackrel{w \rightarrow \infty}{w\log(t)\notin \mathbb{Z}}}\psi_{\chi}^{-}(t^{w})\right)[f(t+0)-f(t-0)]\\
&\Longleftrightarrow &
\alpha [f(t+0)-f(t-0)]=\left(\displaystyle \lim_{\stackrel{w \rightarrow \infty}{w\log(t)\notin \mathbb{Z}}}\psi_{\chi}^{-}(t^{w})\right)[f(t+0)-f(t-0)]\\
&\Longleftrightarrow &
\alpha=\displaystyle \lim_{\stackrel{w \rightarrow \infty}{w\log(t)\notin \mathbb{Z}}}
\psi_{\chi}^{-}(t^{w})\\
&\Longleftrightarrow &
\alpha=\psi_{\chi}^{-}(u),  \hspace{0.5cm} \forall u \in(1, e)\\
&\Longleftrightarrow& (ii)\,\ \mbox{holds}.
\end{eqnarray*}
Let $w\log(t)\notin \mathbb{Z}.$ Since $\psi_{\chi}^{+}$ and $\psi_{\chi}^{-}$ are recurrent functions on $[1,e],$ we have
$$\psi_{\chi}^{+}(t^{w})+\psi_{\chi}^{-}(t^{w})=1.$$ Hence, we obtain
$$(ii)\Longleftrightarrow \psi_{\chi}^{+}(u)=1-\alpha, \hspace{0.5cm} u\in(1, e) .$$
\end{proof}

The above Theorem is proved by assuming that $\chi(u)=0,$ for every $u\in(1,e).$ In the following theorem we prove that without this condition, it is not possible to show the convergence of $I_w^{\chi}f$ at jump discontinuities.

\begin{thm}\label{t4}
 Let $\chi$ be a kernel such that $\psi_{\chi}^{-}(u)=\alpha$ for every $u\in(1,e)$ and $\chi(u)\neq 0,$ for some $u\in(1,e).$  Let $f:\mathbb{R}^{+}\rightarrow\mathbb{R}$ be a bounded measurable signal with a non-removable jump discontinuity $t\in\mathbb{R}^{+}.$ Then $(I_w^{\chi}f)(t)$ does not converge  point-wise at $t$.
\end{thm}

\begin{proof}
Suppose on contrary that $(I_w^{\chi}f)(t)$ converge  point-wise at $t,$ that is $\displaystyle \lim_{\stackrel{w \rightarrow \infty}{w\log(t)\notin \mathbb{Z}}}
(I_w^{\chi}f)(t)=\ell,$ for some $\ell\in\mathbb{R}^{+}.$ Using the existence of the limit and Lemma \ref{lemma1}, we obtain
\begin{eqnarray*}
\ell &=& f(t-0)+\left(\displaystyle \lim_{\stackrel{w \rightarrow \infty}{w\log(t)\notin \mathbb{Z}}}\left[\chi(e^{w\log t-\lfloor w\log t \rfloor} )+\psi_{\chi}^{-}(t^{w})\right]\right)
(f(t+0)-f(t-0)) \\&&-[f(t+0)-f(t-0)] \left(\displaystyle \lim_{\stackrel{w \rightarrow \infty}{w\log(t)\notin \mathbb{Z}}} \chi\left(e^{w\log t-\lfloor w\log t \rfloor} \right) \left(w\log t- \lfloor w\log t \rfloor\right)\right).
\end{eqnarray*}
Since $\psi_{\chi}^{-}(u)=\alpha$ for every $u\in(1,e),$ $\forall w>0, \,\ w\log(t)\notin \mathbb{Z},$ we have
$w\log t- \lfloor w\log t \rfloor=\xi\in(1,e),$ so we obtain
\begin{eqnarray*}
\ell &=& f(t-0)+[f(t+0)-f(t-0)](\alpha+\chi(\xi))-[f(t+0)-f(t-0)]\xi\chi(\xi), \,\ \forall \, \xi\in(1,e).
\end{eqnarray*}
Since $f$ has a jump discontinuity $t\in\mathbb{R}^{+},$ we obtain
\begin{eqnarray*}
\chi(\xi)=\left(\frac{\ell-f(t-0)}{f(t+0)-f(t-0)}-\alpha\right)\left(\frac{1}{1-\xi}\right), \,\ \forall \, \xi\in(1,e).
\end{eqnarray*}
The above expression gives a contradiction. Indeed, if
\begin{eqnarray*}
\frac{\ell-f(t-0)}{f(t+0)-f(t-0)}-\alpha:=C\neq 0,
\end{eqnarray*}
then $\chi(\xi)$ is unbounded on $(1,e),$ so it fails to satisfy condition (i). If $C=0,$ then $\chi(\xi)=0,$ for every $\xi\in(1,e),$ this is again a contradiction. Hence the proof is completed.
\end{proof}

Next we prove a more general theorem on the convergence of the sampling series $I_w^{\chi}f$ for any bounded signal $f$ at jump discontinuities.

\begin{thm}\label{t5}
Let $f:\mathbb{R}^{+}\rightarrow\mathbb{R}$ be a bounded measurable signal and let $t\in\mathbb{R}^{+}$ be a non-removable jump discontinuity point of $f.$ For any $\alpha \in \mathbb{R}$ and any kernel function  $\chi$ satisfying the additional condition $\chi(u)=0,$ for every $u\in[1,e),$ the following are equivalent:
\begin{enumerate}
\item[(i)] $\displaystyle \lim_{w \rightarrow \infty}(I_w^{\chi}f)(t)=\alpha f(t+0)+(1-\alpha)f(t-0),$\\

\item[(ii)] $\psi_{\chi}^{-}(u):=\alpha,$ \hspace{0.5cm} for every $u \in[1, e) $ \\

\item[(iii)]  $\psi_{\chi}^{+}(u):=1-\alpha,$ \hspace{0.5cm} for every $u\in[1, e).$
\end{enumerate}
If we further assume that $\chi$ is continuous on $\mathbb{R}^{+},$ then the above statements are also equivalent to the following:
\begin{enumerate}
\item[(iv)]
\begin{eqnarray*}
\int_{0}^{1} \chi (u) u^{2k\pi i}\frac{du}{u}
=\left\{
\begin{array}{ll}
0 , &  \ \ \ \mbox{if} \ \ \ k \neq 0 \\
\alpha, & \ \ \ \mbox{if} \ \ \ k= 0
\end{array}
\right.
\end{eqnarray*}
\item[(v)]
\begin{eqnarray*}
\int_{1}^{\infty} \chi (u) u^{2k\pi i }\frac{du}{u}
&=&\left\{
\begin{array}{ll}
0 , &  \ \ \ \mbox{if} \ \ \ k \neq 0 \\
1-\alpha, & \ \ \ \mbox{if} \ \ \ k= 0.
\end{array}
\right.
\end{eqnarray*}
\end{enumerate}
\end{thm}

\begin{proof}
Proceeding along the lines proof of Theorem \ref{t2} and Theorem \ref{t3} we see that $(i), (ii)$ and $(iii)$ are equivalent. Now we assume that $\chi$ is continuous on $\mathbb{R}^{+}.$ Define
\begin{eqnarray*}
\chi_0(u)
&:=&\left\{
\begin{array}{ll}
\chi (u) , &  \ \ \ \mbox{for} \ \ \ u<1 \\
0, & \ \ \ \mbox{for} \ \ \ u\geq 1.
\end{array}
\right.
\end{eqnarray*}
Then, we have
\begin{eqnarray*}
\psi_{\chi}^{-}(u)=\sum _{k>\log u}\chi(ue^{-k})=\sum _{k\in\mathbb{Z}}\chi_0(ue^{-k})
\end{eqnarray*}
is a recurrent continuous function with the fundamental interval $[1,e].$ By the Mellin Poisson's summation formula, we get
\begin{eqnarray*}
 \psi_{\chi}^{-}(u)=\sum_{k= - \infty}^{+\infty}\widehat{[\chi_0]_{M}}(2k \pi i) \ u^{-2 k \pi i}
 =\sum_{k= - \infty}^{+\infty}\left(\int_{0}^{1} \chi (u) u^{2k\pi i}\frac{du}{u}\right)u^{-2 k \pi i}.
\end{eqnarray*}
Therefore, we obtain
\begin{eqnarray*}
 \psi_{\chi}^{-}(u)&=&\alpha, \forall u\in[1,e)\\
&\Longleftrightarrow & \widehat{[\chi]_{M}}(2k \pi i)
=\left\{
\begin{array}{ll}
0 , &  \ \ \ \mbox{if} \ \ \ k \neq 0 \\
\alpha, & \ \ \ \mbox{if} \ \ \ k= 0
\end{array}
\right.\\
&\Longleftrightarrow &
\int_{0}^{1} \chi (u) u^{2k\pi i}\frac{du}{u}
=\left\{
\begin{array}{ll}
0 , &  \ \ \ \mbox{if} \ \ \ k \neq 0 \\
\alpha, & \ \ \ \mbox{if} \ \ \ k= 0.
\end{array}
\right.
\end{eqnarray*}
Hence, we have $(ii) \Longleftrightarrow  (iv).$ Using the condition
\begin{eqnarray*}
\sum_{k=- \infty}^{+\infty} \chi(e^{-k}u) =1
\Longleftrightarrow
\widehat{[\chi]_{M}}(2k \pi i)
=\left\{
\begin{array}{ll}
0 , &  \ \ \ \mbox{if} \ \ \ k \neq 0 \\
1, & \ \ \ \mbox{if} \ \ \ k= 0
\end{array}
\right.\\
\end{eqnarray*}
the equivalence between $(iv)$ and $(v)$ can be established easily.
\end{proof}

\begin{remark}
If $f:\mathbb{R}^{+}\rightarrow\mathbb{R}$ be a bounded measurable signal with a removable discontinuity $t\in\mathbb{R}^{+}$, then we have
\begin{eqnarray*}
\displaystyle \lim_{\stackrel{w \rightarrow \infty}{w\log(t)\in \mathbb{Z}}} (I_w^{\chi}f)(t)=\displaystyle \lim_{\stackrel{w \rightarrow \infty}{w\log(t)\notin \mathbb{Z}}} (I_w^{\chi}f)(t)=\displaystyle \lim_{w \rightarrow \infty}(I_w^{\chi}f)(t)=\ell.
\end{eqnarray*}
\end{remark}

\section{Degree of Approximation}
We study the degree of approximation of $I_{w}^{\chi}f$ by using the logarithmic modulus of continuity. A function $f: \mathbb{R}^+ \rightarrow \mathbb{R}$ is said to be
log-uniformly continuous if for every $\epsilon > 0,$ $\exists\,\ \delta > 0$ such that $|f(s) -f(t)| < \epsilon$ whenever $| \log s - \log t| < \delta,$ for any $s, t \in \mathbb{R}^{+}.$ The set of all bounded log-uniformly continuous functions is $\mathcal{C} (\mathbb{R}^+).$  It is easy to see that $\mathcal{C} (\mathbb{R}^+)$ is contained in $C(\mathbb{R}^+),$ where $C(\mathbb{R}^+)$ is the set of all bounded continuous functions on $\mathbb{R}^+$ with usual sup norm $\|f\|_{\infty} := \sup_{x \in \mathbb{R}^+} |f(x)|.$ For $f\in C(\mathbb{R}^{+}),$ the logarithmic modulus of continuity is defined by
$$ \omega(f,\delta):= \sup \{|f(s)-f(t)|:\  \mbox{whenever} \  |\log (s)-\log (t)| \leq \delta,\  \ \delta \in \mathbb{R}^{+}\} .$$
$\omega(f,\delta)$ has the following properties:
\begin{itemize}
\item[(a)] $\omega(f, \delta) \rightarrow 0,$ as $\delta\rightarrow 0.$

\item[(b)] $\displaystyle |f(s)-f(t)|\leq  \omega(f,\delta)\left(1+\dfrac{|\log s-\log t|}{\delta}\right).$
\end{itemize}

Now we show the order of convergence for the Kantorovich exponential sampling series when $M_{\nu}(\chi)<+\infty$ for $0<\nu<1.$

\begin{thm} \label{t7}
Let $\chi$ be a kernel such that $M_{\nu}(\chi)<+\infty$ for $0<\nu<1.$ Then for any $ f \in \mathcal{C} (\mathbb{R}^+)$ and for sufficiently large $w>0,$ we have
 $$ |(I_{w}^{\chi}f)(t) - f(t)| \leq  \omega (f,w^{-\nu})[M_{\nu}(\chi)+2M_{0}(\chi)]+2^{\nu+1}\|f\|_{\infty}M_{\nu}(\chi)w^{-\nu},$$
for every $t\in\mathbb{R}^{+}.$
\end{thm}

\begin{proof}
Let $t\in\mathbb{R}^{+}$ be fixed. Then using the condition $\displaystyle\sum_{k= - \infty}^{+\infty}  \chi(e^{-k} t^w) =1,$ we obtain
\begin{eqnarray*}
 |(I_{w}^{\chi}f)(t) - f(t)| &=& \left| \sum_{k= - \infty}^{+\infty} \chi(e^{-k} t^{w})w\int_{\frac{k}{w}}^{\frac{k+1}{w}}(f(e^{u}) - f(t))du \right| \\
&\leq& \left( \sum_{\big|k-w \log t\big|\leq \frac{w}{2}} +\sum_{\big|k- w\log t\big| > \frac{w }{2}} \right) \big| \chi(e^{-k} t^{w})\big| w\int_{\frac{k}{w}}^{\frac{k+1}{w}}|f(e^{u}) - f(t)|du \\
&:=& I_{1}+I_{2}.
\end{eqnarray*}
Let $u\in \left[\frac{k}{w}, \frac{k+1}{w}\right].$ If $| w\log t-k|\leq \frac{w}{2}$, we have
$|u-\log t|\leq\left|u-\frac{k}{w}\right|+\left|\frac{k}{w}-\log t\right|\leq \frac{1}{w}+\frac{1}{2}\leq 1.$ Let $0<\nu<1$ and for every $w>0$ sufficiently large,  we have
$$\omega\left(f,\left |u- \log t \right|\right)\leq\omega\left(f,\left|u - \log t \right|^{\nu}\right). $$ Therefore, using the above inequality and the property $(b),$ we obtain
\begin{eqnarray*}
I_{1}&\leq&\sum_{\left|k-w \log t\right|\leq \frac{w }{2}}\left| \chi(e^{-k} t^{w})\right| w\int_{\frac{k}{w}}^{\frac{k+1}{w}}\omega\left(f,\left |u- \log t \right|^{\nu}\right)du\\
&\leq&\sum_{\left|k-w \log t\right|\leq\frac{w }{2}}\left| \chi(e^{-k} t^{w})\right| w \int_{\frac{k}{w}}^{\frac{k+1}{w}}\left(1+w^{\nu}\left |u- \log t \right|^{\nu}\right)\omega(f,w^{-\nu})du\\
&\leq&\omega(f,w^{-\nu})\left(\sum_{\left|k-w \log t\right|\leq \frac{w }{2}}\left| \chi(e^{-k} t^{w})\right| w\int_{\frac{k}{w}}^{\frac{k+1}{w}}w^{\nu}\left |u- w\log t \right|^{\nu}du+\sum_{\left|k-w \log t\right|\leq\frac{w }{2}}\left| \chi(e^{-k} t^{w})\right| \right)\\
&=:&I_1^{'}+I_{2}^{''}.
\end{eqnarray*}
Using the condition $(ii)$ we easily obtain $I_{2}^{''}\leq m_{0}(\chi).$ Utilizing the sub-additivity of $|.|^{\nu}$ for $0<\nu<1$, we have:
\begin{eqnarray*}
I_1^{'}&\leq&\sum_{\left|k-w \log t\right|\leq \frac{w}{2}}\left| \chi(e^{-k} t^{w})\right| \left (w^{\nu}  \max_{u\in\left[\frac{k}{w}, \frac{k+1}{w}\right]}\left |u- \log t \right|^{\nu}\right)\\
&\leq&\sum_{\left|k-w \log t\right|\leq \frac{w}{2}}\left| \chi(e^{-k} t^{w})\right| w^{\nu}\left(  \max\left(\left |\frac{k}{w}- \log t \right|^{\nu},\left |\frac{k+1}{w}- \log t \right|^{\nu}\right)\right)\\
&\leq&\sum_{\left|k-w \log t\right|\leq \frac{w}{2}}\left| \chi(e^{-k} t^{w})\right| \left(w^{\nu} \max\left(\left |\frac{k}{w}- \log t \right|^{\nu},\left |\frac{k}{w}- \log t \right|^{\nu}+\frac{1}{w^{\nu}}\right)\right)\\
&\leq&\sum_{\left|k-w \log t\right|\leq\frac{w}{2}}\left| \chi(e^{-k} t^{w})\right| \left(w^{\nu}  \left(\left|\frac{k}{w}- \log t \right|^{\nu}+\frac{1}{w^{\nu}}\right)\right)\\
&\leq& M_{\nu}(\chi)+M_{0}(\chi)<+\infty.
\end{eqnarray*}
Finally we estimate $I_{2}$:
\begin{eqnarray*}
I_{2} &\leq&2\|f\|_{\infty} \sum_{\big|k- w\log t\big| > \frac{w}{2}}  \big| \chi(e^{-k} t^{w})\big|\leq
2\|f\|_{\infty} \sum_{\big|k- w\log t\big| > \frac{w }{2}} \frac{\big|k- w\log t\big|^{\nu} }{\big|k- w\log t\big| ^{\nu}} \big| \chi(e^{-k} t^{w})\big|\\
&\leq&2^{\nu+1}\|f\|_{\infty}w^{-\nu}\sum_{\big|k- w\log t\big| > \frac{w }{2}} \big| \chi(e^{-k} t^{w})\big|\big|k- w\log t\big|^{\nu} \leq
2^{\nu+1}\|f\|_{\infty}w^{-\nu}M_{\nu}(\chi)<+\infty.
\end{eqnarray*}
On combining the estimates $I_{1}$ and $I_{2},$ we get the desired result.
\end{proof}

\section{Linear Prediction of Signals}
The Kantorovich exponential sampling series can be used for predicting the signals at a time $t$ using the sample values taken in the past. This is addressed in the following theorem.

\begin{thm}
Let $\chi$ be a kernel with compact support and $f:\mathbb{R}^{+}\rightarrow \mathbb{R}$ be a bounded measurable signals and $w>0.$ If $supp\, \chi \subset (1,\infty),$ then for every $w>0$ and $t\in\mathbb{R}^{+}$ with $w\log t \in \mathbb{Z},$ we have
$$ (I_{w}^{\chi}f)(t)= \sum_{k<w\log t} \chi(e^{-k}t^{w})w\int_{\frac{k}{w}}^{\frac{k+1}{w}}f(e^{u})du.$$
Further, if $supp\, \chi \subset (e,\infty),$ then for every $w>0$ and $t\in\mathbb{R}^{+}$ with $w\log t \notin \mathbb{Z},$ there holds
$$(I_{w}^{\chi}f)(t)=\sum_{k\leq \lfloor w\log t\rfloor -1}  \chi(e^{-k}t^{w})w\int_{\frac{k}{w}}^{\frac{k+1}{w}}f(e^{u})du.$$
\end{thm}

\begin{proof}
Let $w>0.$ As $supp\,\chi$ $\subset (1,\infty),$ we obtain $ \chi(e^{-k}t^{w})=0,$ for every $k\in\mathbb{Z}$ such that $e^{-k}t^{w}\leq 1,$ this implies that $\log t \leq \frac{k}{w}.$ If $w \log t \in \mathbb{Z}$, it is easy to observe that the last term of the Kantorovich exponential sampling series is
\begin{eqnarray*}
\chi(e)\left(w\int_{\log t-\frac{1}{w}}^{\log t}f(e^{u})du\right).
\end{eqnarray*}
Therefore, we see that Kantorovich exponential sampling series exploit the values of the signal $f$ in the past with respect to fixed time $t.$ \\
\indent Now the second case can be estimated as follows: Let $supp\, \chi \subset (e,\infty).$ If $w\log t \notin \mathbb{Z}$ then $\lfloor w\log t\rfloor <w\log t<\lfloor w\log t\rfloor+1$. Finally, we have
\begin{eqnarray*}
 \sum_{k<w\log t} \chi(e^{-k}t^{w})w\int_{\frac{k}{w}}^{\frac{k+1}{w}}f(e^{u})du&=&\sum_{k\leq \lfloor w\log t\rfloor -1}\chi(e^{-k}t^{w})w\int_{\frac{k}{w}}^{\frac{k+1}{w}}f(e^{u})du
\\&& +\chi(e^{w\log t-\lfloor w\log t\rfloor})\left(w\int_{\frac{\lfloor w\log t\rfloor}{w}}^{\frac{\lfloor w\log t\rfloor+1}{w}}f(e^{u})du\right).
\end{eqnarray*}
Since $\chi(e^{w\log t-\lfloor w\log t\rfloor})=0$ if $0<w\log t-\lfloor w\log t\rfloor <1,$ we get
$$(I_{w}^{\chi}f)(t)=\sum_{k\leq \lfloor w\log t\rfloor -1}  \chi(e^{-k}t^{w})w\int_{\frac{k}{w}}^{\frac{k+1}{w}}f(e^{u})du.$$ Hence the proof is completed.
\end{proof}

\begin{remark}
It is observed that in order to approximate the signal at a time $t,$ the local average values of the signals are computed before the time instance $t.$ The above theorem can be used for predicting the future signal values from the past value of the signal.
\end{remark}

\section{Examples of the Kernels and Special Cases}
\subsection{Mellin-B spline kernels}
The Mellin B-splines of order $n$  for $ x \in \mathbb{R}^{+}$ are given by
$$\bar{B}_{n}(x):= \frac{1}{(n-1)!} \sum_{j=0}^{n} (-1)^{j} {n \choose j} \bigg( \frac{n}{2}+\log x-j \bigg)_{+}^{n+1}.$$
$\bar{B}_{n}(x) $ is compactly supported for every $ n \in \mathbb{N}.$ So it is bounded and it belongs to $L^{1}(\mathbb{R}^{+}).$ Therefore condition (i) is satisfied. The Mellin transform of $\bar{B}_{n}$ (see \cite{bardaro7}) is
\begin{eqnarray*}
\widehat{[\bar{B}_{n}]_{M}}(c+it) = \bigg( \frac{\sin(\frac{t}{2})}{(\frac{t}{2})} \Bigg)^{n},  \ \ \hspace{0.5cm} t \neq 0.
\end{eqnarray*}
Using the Mellin's-Poisson summation formula (see \cite{bardaro7}), it is easy to see that $\bar{B}_{n}(x)$ satisfies the condition (ii). The condition (iii) is also easily verified.

\subsection{Mellin Jackson kernels}
For $y\in\mathbb{R}^{+}, \beta\in\mathbb{N}, \gamma\geq 1,$ the Mellin Jackson kernels are defined by
$$J_{\gamma, \beta}^{-1}(y):=d_{\gamma, \beta}\,\ y^{-c} sinc^{2\beta}\left(\frac{\log y}{2\gamma\beta\pi}\right),$$
where
$$d_{\gamma, \beta}^{-1}:=\int_{0}^{\infty}sinc^{2\beta}\left(\frac{\log y}{2\gamma\beta\pi}\right)\frac{du}{u}.$$ The Mellin Jackson kernels also satisfy the conditions (i),(ii) and (iii) (see \cite{bardaro7}). We can study the convergence of the $I_{w}^{\chi}f$ with jump discontinuity associated with these kernels only for the case given in Theorem \ref{t2} and we observe that $\chi(u)\neq 0,$ for every $u\in (1,e)$. Hence, Theorem \ref{t3} and Theorem \ref{t5} can not be applied for these kernels. In order to find examples of the kernels for which the exponential sampling series converge at any jump discontinuity $t\in\mathbb{R}^{+}$ of the given bounded signal $f:\mathbb{R}^{+}\rightarrow \mathbb{R}$, we need to construct suitable kernels. One such construction is given below.

\begin{thm}\label{t11}
Let $a,b\in\mathbb{R}^{+}$ and let $\chi_{a},\chi_{b}$ be two continuous kernels satisfying conditions (i), (ii) and (iii) such that $supp\, \chi_{a}\subseteq {[e^{-a},e^{a}]}$ and $supp\, \chi_{b}\subseteq {[e^{-b},e^{b}]}.$ For a fixed $\alpha\in\mathbb{R},$ define $\chi:\mathbb{R}^{+}\rightarrow \mathbb{R}$ by
$$\chi{(u)}:=(1-\alpha)\chi_{a}(ue^{-a-1})+\alpha\chi_{b}(ue^{b}), \,\,\ u\in\mathbb{R}^{+}.$$
Then $\chi$ also satisfies conditions (i), (ii) and (iii) and in addition $\chi(u)=0,$ for every $u\in[1,e).$ Further, the Kantorovich exponential sampling series $I_w^{\chi}{f}, w>0$ corresponding to $\chi$ satisfying (i) of Theorem \ref{t5} at any jump discontinuity $t\in\mathbb{R}^{+}$ of the given bounded measurable signal $f$ with $\alpha$ as a parameter.
\end{thm}

\begin{proof}
We estimate the Mellin transform of $\chi(u)$ as follows:
\begin{eqnarray*}
\widehat{[\chi]_{M}}(s)
&=&\int_{0}^{\infty}(1-\alpha)\chi_{a}(te^{-a-1})t^{c+iv-1}dt+\int_{0}^{\infty}\alpha\chi_{b}(te^{b})t^{c+iv-1}dt
:=I_{1}+I_{2}.
\end{eqnarray*}
Making substitution $u=e^{-a-1},\,\ du=e^{-a-1}dt$ in $I_1$ and $u=te^{b}, \,\ du=e^{b}dt$ in $I_2,$ we easily obtain
\begin{eqnarray*}
\widehat{[\chi]_{M}}(s)=(1-\alpha)\widehat{[\chi_{a}]_{M}}(s)e^{(1+a)s}+\alpha\widehat{[\chi_{b}]_{M}}(s)e^{-bs},
\end{eqnarray*}
where $s=c+iv.$ It is easy to check that $\chi$ satisfies conditions (i) and (iii). Now we show that kernel satisfies the condition (ii). We obtain
\begin{eqnarray*}
\widehat{[\chi]_{M}}(2k\pi i )&=&(1-\alpha)\widehat{[\chi_{a}]_{M}} (2k\pi i)e^{(1+a)(2k\pi i)}+\alpha\widehat{[\chi_{b}]_{M}}(2k\pi i)e^{-b(2k \pi i)}.
\end{eqnarray*}
Since $\chi_{a}$ and $\chi_{b}$ satisfies condition (ii), we have
\begin{eqnarray*}
\widehat{[\chi_{a}]_{M}} (2k\pi i )=\widehat{[\chi_{b}]_{M}} (2k\pi i )= \left\{
\begin{array}{ll}
0 , &  \ \ \ \mbox{if} \ \ \ k \neq 0 \\
1, & \ \ \ \mbox{if} \ \ \ k= 0.
\end{array}
\right.
\end{eqnarray*}
For suitable choice of $a$ and $b$, we obtain
\begin{eqnarray*}
\widehat{[\chi]_{M}} (2k\pi i )= \left\{
\begin{array}{ll}
0 , &  \ \ \ \mbox{ if} \ \ \ k \neq 0 \\
1, & \ \ \ \mbox{if} \ \ \ k= 0.
\end{array}
\right.
\end{eqnarray*}
Hence, $\chi$ satisfies condition (ii) and we easily see that $ \chi $ also satisfies $\chi(u)=0$ for $1\leq u<e. $
Now, we obtain
\begin{eqnarray*}
\int_{1}^{\infty} \chi (u) u^{2k\pi i }\frac{du}{u}&=&(1-\alpha)\int_{1}^{\infty} \chi_{a}(ue^{-a-1})u^{2k\pi i}\frac{du}{u}\\
&=& (1-\alpha)e^{2k\pi i(1+a)}\widehat{[\chi_{a}]_{M}}(2k\pi i )\\
&=&\left\{
\begin{array}{ll}
0 , &  \ \ \ \mbox{if} \ \ \ k \neq 0 \\
1-\alpha, & \ \ \ \mbox{if} \ \ \ k= 0.
\end{array}
\right.
\end{eqnarray*}
Therefore, the condition (v) of Theorem is satisfied, hence the proof is completed.
\end{proof}

As a consequence of the following theorem, condition on the kernels can be given so that $ M_{\nu}(\chi)<+\infty$ for some $0\leq \nu<q,$ for some $q<1$ and
 $M_{\nu}(\chi)=+\infty$ for $q<\nu\leq 1.$

\begin{thm}\label{t12}
Let $\chi : \mathbb{R}^{+} \rightarrow \mathbb{R}$ be a function such that
\begin{eqnarray*}
\frac{C_{1}}{|\log u|^{p}} \leq \chi({u})\leq \frac{C_{2}}{|\log u|^{p}}, \hspace{1cm} \forall \,\ |\log u|>M
\end{eqnarray*}
holds for suitable constants $0<C_{1}\leq C_{2}$ and for some $1<p \leq 2,$ and $M >0$. Then, we have
\begin{eqnarray*}
 M_{\nu}(\chi)= \left\{
\begin{array}{ll}
+\infty , &  \ \ p -1\leq \nu \leq 1 \\
<+\infty, & \ \ \ 0\leq \nu <p-1.
\end{array}
\right.
\end{eqnarray*}
\end{thm}

\begin{proof}
Let $p-1\leq \nu \leq 1,$ and $u \in \mathbb{R^{+}}$ be fixed. Since $p-\nu\leq1,$ we have
\begin{eqnarray*}
M_{\nu}(\chi) &\geq &\sum_{k=-\infty}^{k=+\infty}|\chi(e^{-k}u)| |\log u-k|^{\nu}\\
&\geq &\sum_{k=-\infty}^{k=+\infty}|\chi(e^{-k}u)| |\log u-k|^{\nu}\frac{|\log u-k|^{p}}{|\log u-k|^{p}}\\
&\geq & C_{1}\sum_{|\log u-k|>M}|\log u-k|^{\nu-p}=+\infty.
\end{eqnarray*}
Let $0\leq \nu<p-1$ and $u\in \mathbb{R^{+}}.$ Since $M_{0}(\chi)<+\infty$ and $\chi(u)=O(|\log u|^{-p})$, as $|u|\rightarrow +\infty$ with $p>1$ and using the fact that
the series$$ \sum_{|\log u-k|>M}|\log u-k|^{\nu-p}$$ with $p-\nu >1$, is uniformly convergent for every $u\in \mathbb{R^{+}},$ we obtain
\begin{eqnarray*}
\sum_{k=-\infty}^{k=+\infty}|\chi(e^{-k}u)| |\log u-k|^{\nu}&\leq& \left(\sum_{|\log u-k|\leq M}+\sum_{|\log u-k|>M}\right)|\chi(e^{-k}u)| |\log u-k|^{\nu}\\
&\leq &M^{\nu}M_{0}(\chi)+C_{2}\sum_{|\log u-k|>M}|\log u-k|^{\nu-p}<+\infty.
\end{eqnarray*}
Hence, the proof is completed.
\end{proof}

\subsection{Implementation Results}

We show the approximation of discontinuous function
\begin{eqnarray*}
f(t)&=&\left\{
\begin{array}{ll}
\dfrac{2}{1+et} , &  \ \ \ \ \ 0<t<\dfrac{1}{e} \\
\vspace{0.25 cm}
2, & \ \ \  \ \ \ \dfrac{1}{e}\leq t< e \\
\vspace{0.25 cm}

3, & \ \ \  \ \ \ \ e\leq t< 4 \\
\vspace{0.25 cm}

\dfrac{10}{4+t}, & \ \ \  \ \ \ t\ge 4
\end{array}
\right.
\end{eqnarray*}
by $I_w^{\chi}{f}$ at jump discontinuities at $t=\dfrac{1}{e},$ $t=e$ and $t=4.$ We define
\begin{eqnarray*}
\chi_c(t)=\dfrac{2}{5}\bar{B}_{2}(te^{-2})+\dfrac{3}{5}\bar{B}_{2}(te),
\end{eqnarray*}
where the kernel $\bar{B}_{2}$ is defined by
\begin{eqnarray*}
\bar{B}_{2}(x)&=&\left\{
\begin{array}{ll}
1-\log x, &  \ \ \ \ \ 1<x<e \\
\vspace{0.25 cm}
1+\log x, & \ \ \  \ \ \ \dfrac{1}{e}<x<1
\\
\vspace{0.25 cm}
0, & \ \ \  \ \ \ \mbox{otherwise.}
\end{array}
\right.
\end{eqnarray*}
It is easy to verify that $\chi_c(t)$ satisfies the conditions (i), (ii), (iii) and $\chi_c(u)=0,$ for every $u\in[1,e).$ Further we observe that the condition (i) of Theorem \ref{t5} is also satisfied with $\alpha=\dfrac{3}{5}.$ In view of Theorem \ref{t5} and Theorem \ref{t11}, we have
$I_w^{\chi_c}{f}\rightarrow \dfrac{3}{5}f(t+0)+\dfrac{2}{5}f(t-0)$ as $w\rightarrow \infty$ at a jump discontinuity points $t==\dfrac{1}{e},$ $t=e$ and $t=4.$

Next we consider discontinuous kernels: $\chi_{d}(t):=\chi_c(t)+\theta(t),$ $t\in\mathbb{R}^{+},$ where
\begin{eqnarray*}
\chi_c(t)=\dfrac{2}{5}\bar{B}_{2}(te^{-2})+\dfrac{3}{5}\bar{B}_{2}(te).
\end{eqnarray*}
and $\theta(t)$ is defined by
\begin{eqnarray*}
\theta (t)&=&\left\{
\begin{array}{ll}
1, &  \ \ \ \ \ t=e,\dfrac{1}{e}\\
\vspace{0.25 cm}

-1, &  \ \ \ \ \ t=e^2,\dfrac{1}{e^2}\\

\vspace{0.25 cm}
0, & \ \ \  \ \ \ \mbox{otherwise.}
\end{array}
\right.
\end{eqnarray*}
Then, we observe that $\displaystyle\sum_{k\in \mathbb{Z}}\theta(e^{-k}u)=0,$ and $ \psi_{\chi}^{-}(u)=0$ for every $u\in[1,e).$ Also we observe that $\chi_{d}(u)$ is not necessarily a continuous (see Fig.4) and it satisfies the condition (i), (ii), (iii) and $\chi_{d}(u)=0,$ for every $u\in[1,e).$ Again from Theorem \ref{t5} and Theorem \ref{t11}, we see that $I_w^{\chi_{d}}{f}$ converges to $\dfrac{3}{5}f(t+0)+\dfrac{2}{5}f(t-0)$ as $w\rightarrow \infty$ at a jump discontinuity points $t==\dfrac{1}{e},$ $t=e$ and $t=4.$ The convergence of $I_w^{\chi_c}{f}$ and $I_w^{\chi_{d}}{f}$ at discontinuity points $t=\dfrac{1}{e},$ $t=e$ and $t=4$ of the function $f$ has been tested and numerical results have been provided in Tables 1, 2 and 3.

\begin{figure}[h]
\centering
{\includegraphics[width=0.8\textwidth]{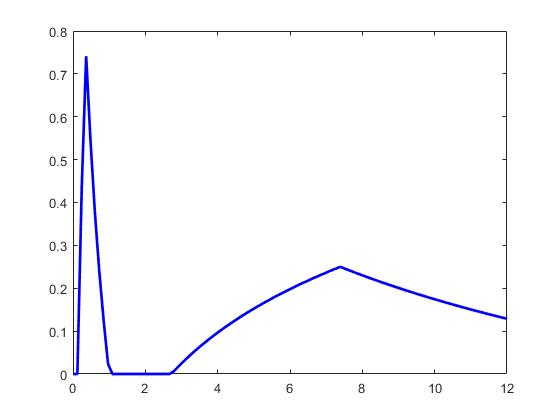}}
\caption{Plot of the kernel $\chi_c(t)=\dfrac{2}{5}\bar{B}_{2}(te^{-2})+\dfrac{3}{5}\bar{B}_{2}(te).$}
\end{figure}

\begin{figure}[h]
\centering
{\includegraphics[width=0.8\textwidth]{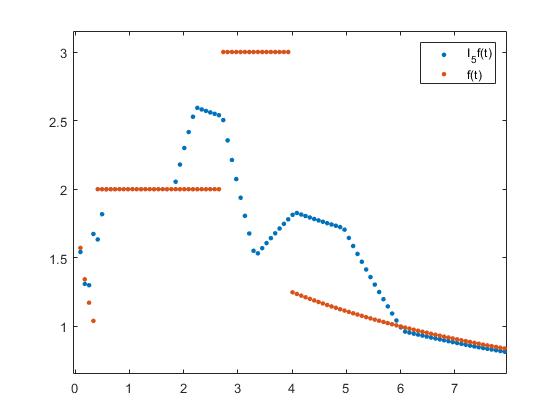}}
\caption{Approximation of $f(t)$ by $I_w^{\chi}{f}$ based on $\chi_c(t)=\dfrac{2}{5}\bar{B}_{2}(te^{-2})+\dfrac{3}{5}\bar{B}_{2}(te)$ for $w=5.$}
\end{figure}

\begin{figure}[h]
\centering
{\includegraphics[width=0.8\textwidth]{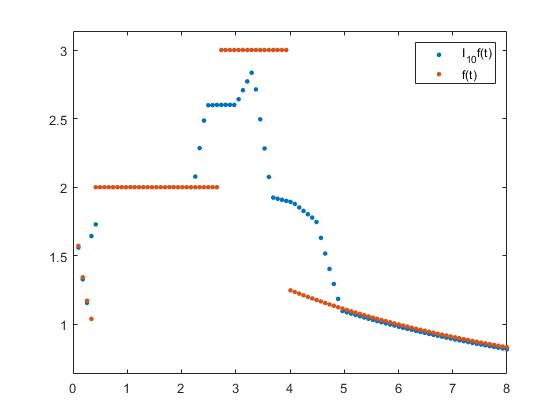}}
\caption{Approximation of $f(t)$ by $I_w^{\chi}{f}$ based on $\chi_c(t)=\dfrac{2}{5}\bar{B}_{2}(te^{-2})+\dfrac{3}{5}\bar{B}_{2}(te)$ for $w=10.$}
\end{figure}

\clearpage

\begin{tb}\label{table1}
 {\it Approximation of $f$ at the jump discontinuity point $t=\dfrac{1}{e}$ by the Kantorovich exponential sampling series $I_w^{\chi_c}{f}$ and $I_w^{\chi_{d}}{f}$ based on $\chi_c(t),$ $\chi_{d}(t)$ for different values of $w>0.$ The theoretical limit is:
 $$\lim_{w\rightarrow \infty}(I_w^{\chi_c}{f})\left(\dfrac{1}{e}\right)=\lim_{w\rightarrow \infty}(I_w^{\chi_{d}}{f})\left(\dfrac{1}{e}\right)
 =\dfrac{3}{5}f\left(\dfrac{1}{e}+0\right)+\dfrac{2}{5}f\left(\dfrac{1}{e}-0\right)=1.60.$$}
\begin{center}
\begin{tabular} {|l|l|l|l|l|l|l|l|l|l|l|}\hline
 $w$ & $ 1$ & $2$ & $3$ & $4$ & $ 5$ & $10$ & $20$ & $50$ & $100$ & $200$\\
 \hline
 $I_w^{\chi_c}{f}$&$1.8509$ & $1.7427$ & $1.6978$ & $1.6740$ &$1.6595$ & $1.6299$ & $1.6150$ & $1.6060$ & $1.6030$ & $1.6015$\\
  \hline
 $I_w^{\chi_d}{f}$&$1.6610$ & $-0.4903$ & $2.7807$ & $-0.4487$ &$1.6595$ & $1.6299$ & $1.6150$ & $1.6060$ & $1.6030$ & $1.6015$\\
   \hline
\end{tabular}
\end{center}
   \end{tb}

\begin{tb}\label{table2}
 {\it Approximation of $f$ at the jump discontinuity point $t={e}$ by the Kantorovich exponential sampling series $I_w^{\chi_c}{f}$ and $I_w^{\chi_{d}}{f}$ based on $\chi_c(t),$ $\chi_{d}(t)$ for different values of $w>0.$ The theoretical limit is:
 $$\lim_{w\rightarrow \infty}(I_w^{\chi_c}{f})(e)=\lim_{w\rightarrow \infty}(I_w^{\chi_{d}}{f})(e)
 =\dfrac{3}{5}f(e+0)+\dfrac{2}{5}f(e-0)=2.60.$$}
\begin{center}
\begin{tabular}{|l|l|l|l|l|l|l|l|l|l|l|}\hline
 $w$ & $ 1$ & $2$ & $3$ & $4$ & $ 5$ & $10$ & $20$ & $50$ & $100$ & $200$\\
 \hline
 $I_w^{\chi_c}{f}$&$1.1746$ & $1.4158$ & $1.6687$ & $2.1126$ &$2.5316$ & $2.600$ & $2.600$ & $2.600$ & $2.600$ & $2.600$\\
  \hline
 $I_w^{\chi_d}{f}$&$2.8973$ & $0.6722$ & $1.6687$ & $2.1126$ &$2.5316$ & $2.600$ & $2.600$ & $2.600$ & $2.600$ & $2.600$\\
   \hline
\end{tabular}
\end{center}
   \end{tb}

\begin{tb}\label{table3}
 {\it Approximation of $f$ at the jump discontinuity point $t=4$ by the Kantorovich exponential sampling series $I_w^{\chi_c}{f}$ and $I_w^{\chi_{d}}{f}$ based on $\chi_c(t),$ $\chi_{d}(t)$ for different values of $w>0.$ The theoretical limit is:
 $$\lim_{w\rightarrow \infty}(I_w^{\chi_c}{f})(4)=\lim_{w\rightarrow \infty}(I_w^{\chi_{d}}{f})(4)
 =\dfrac{3}{5}f(4+0)+\dfrac{2}{5}f(4-0)=1.95.$$}
\begin{center}
\begin{tabular}{|l|l|l|l|l|l|l|l|l|l|l|}\hline
 $w$ & $ 1$ & $2$ & $3$ & $4$ & $ 5$ & $10$ & $20$ & $50$ & $100$ & $200$\\
 \hline
 $I_w^{\chi_c}{f}$&$1.0941$ & $1.2847$ & $1.4307$ & $1.6287$ &$1.8092$ & $1.8939$ & $1.9219$ & $1.9388$ & $1.9444$ & $1.9472$\\
  \hline
 $I_w^{\chi_d}{f}$&$1.0941$ & $1.2847$ & $1.4307$ & $1.6287$ &$1.8092$ & $1.8939$ & $1.9219$ & $1.9388$ & $1.9444$ & $1.9472$\\
   \hline
\end{tabular}
\end{center}
   \end{tb}

\begin{figure}[h]
\centering
{\includegraphics[width=0.8\textwidth]{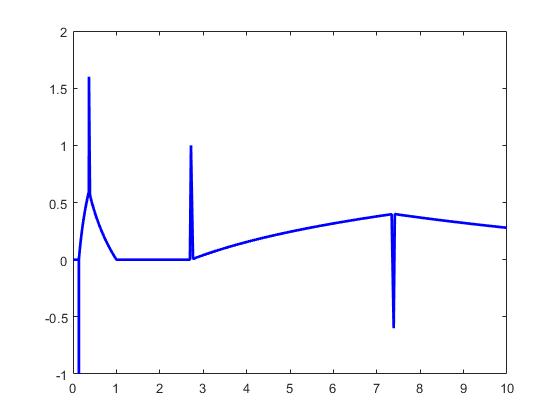}}
\caption{Plot of the kernel $\chi_{d}(t)=\chi_c(t)+\theta(t).$}
\end{figure}

\begin{figure}[h]
\centering
{\includegraphics[width=0.8\textwidth]{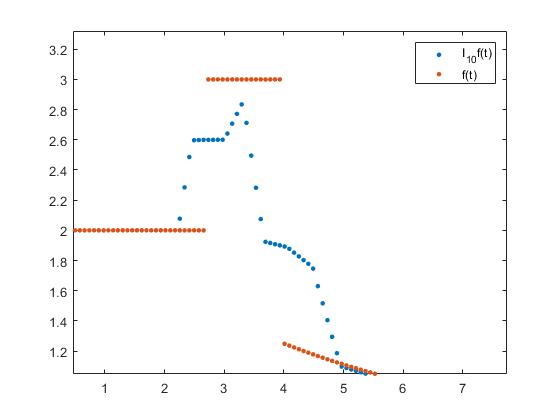}}
\caption{Approximation of $f(t)$ by $I_w^{\chi}{f}$ based on $\chi_{d}(t)=\chi_c(t)+\theta(t)$ for $w=10.$}
\end{figure}

\clearpage

{\bf Acknowledgments.} The first and second author is supported by DST-SERB, India Research Grant EEQ/2017/000201. The third author P. Devaraj has been supported by DST-SERB Research Grant MTR/2018/000559.

\end{document}